\documentclass{amsart}
\usepackage[english]{babel}
\usepackage{epsfig}

\usepackage{amsmath}
\usepackage{amssymb}
\usepackage{amsthm}
\usepackage{amsfonts}
\usepackage{color}
\usepackage{amssymb}
\usepackage{mathrsfs}

\newtheorem{theorem}{Theorem}[section]
\newtheorem{definition}[theorem]{Definition}
\newtheorem{defn}[theorem]{Definition}
\newtheorem{lemma}[theorem]{Lemma}
\newtheorem{lem}[theorem]{Lemma}
\newtheorem{proposition}[theorem]{Proposition}
\newtheorem{pro}[theorem]{Proposition}
\newtheorem{corollary}[theorem]{Corollary}
\newtheorem{oss}[theorem]{Remark}

\newcommand{\ff}{\mathcal{O}}
\newcommand{\C}{\mathbb{C}}

\DeclareMathOperator{\IIm}{Im}
\newcommand{\HH}{\mathbb{H}}
\DeclareMathOperator{\RRe}{Re}
\DeclareMathOperator{\syz}{syz}
\DeclareMathOperator{\ext}{ext}
\newcommand{\hh}{{\mathbb{H}}}
\newcommand{\s}{{\mathbb{S}}}

\newcommand{\rr}{{\mathbb{R}}}
\newcommand{\nn}{{\mathbb{N}}}

\newcommand{{\ee}}{{\`e}}
\newcommand{{\aA }}{{\`a}}
\newcommand{{\oo}}{{\`o}}
\newcommand{{\uu}}{{\`u}}
\newcommand{{\ii}}{{\`i}}

\title{Ideals of regular functions of a quaternionic variable}

\author{Graziano Gentili}
\address{Graziano Gentili \\ Dipartimento di Matematica e Informatica ``U. Dini'' \\ Universit\`a di Firenze \\
Viale Morgagni 67/A, 50134 Firenze, Italy}
\email{gentili@math.unifi.it}

\author{Giulia Sarfatti}
\address{Giulia Sarfatti \\ Dipartimento di Matematica e Informatica ``U. Dini'' \\ Universit\`a di Firenze \\
Viale Morgagni 67/A, 50134 Firenze, Italy}
\email{sarfatti@math.unifi.it}

\author{Daniele C. Struppa}
\address{Daniele C. Struppa \\ Schmid College of Science \& Technology \\ Chapman University \\
One University Drive, Orange, CA 92866, USA}
\email{struppa@chapman.edu}

\begin{document}

\begin{abstract} In this paper we prove that, for any $n\in \nn$, the ideal generated by $n$ slice regular functions   $f_1,\ldots,f_n$ having no common zeros coincides with the entire ring of slice regular functions. The proof required the study of the non-commutative syzygies of a vector of regular functions, that manifest a different character when compared with their complex counterparts.
%This result was proved under stronger hypotheses, which are now not necessary anymore, in \cite{sarfatti}. The proof relies on a careful study of the local behavior of slice regular functions, and on some  features of the syzygies of three or more such functions.

%In particular we present in full detail the case $n=2$, and we use the case $n=3$ to illustrate unexpected features of the sygyzies of three or more slice regular functions.


\end{abstract}
\maketitle

\section{Introduction}

The theory of slice regular functions of a quaternionic variable (often simply called regular functions) has been introduced in \cite{CRAS}, \cite{3}, and further developed in a series of papers, including in particular \cite{advances2}, where most of the recent developments are discussed.  The  full theory is presented in the monograph \cite{libroGSS}, while an extension of the theory to the case of real alternative algebras is discussed in \cite{ghiloniperotti0}, \cite{ghiloniperotti1} and \cite{ghiloniperotti2}.
The theory of regular functions has been applied to the study of a non-commutative functional calculus, (see for example the monograph \cite{librofc} and the references therein) and to address the problem of the construction and classification of orthogonal complex structures in open subsets of the space $\hh$ of quaternions (see \cite{GenSalSto}).
%The definition of regular function we use in this paper is the following.
%\begin{definition}
%Let $\Omega$ be a domain in $\hh$, and let $\mathbb{S}$ denote the $2-$sphere of imaginary units in $\hh$. A function $f: \Omega \rightarrow \hh$ is called {\em regular} if, for all $I\in \mathbb{S}$,  its restriction $f_I$ to $\Omega_I:=\Omega \cap (\rr + \rr I)$ is \emph{holomorphic}, i.e, if it has continuous partial derivatives and satisfies
%$$\overline{\partial_I} f(x+yI):=\frac{1}{2}\Big(\frac{\partial}{\partial x}  +I\frac{\partial}{\partial y}\Big)f_I(x+yI)=0$$
%for every $x+yI \in \Omega_I$.
%\end{definition}
%As shown in \cite{3}, if $\Omega$ is an open ball centered at a real point of $\hh$, the class of regular functions coincides with the class of convergent power series of type $\sum_{n=0}^{+\infty} q^na_n$, with all $a_n\in \hh$.
%
In many cases, the results one obtains in the theory of regular functions are inspired by complex analysis, though they often require essential modifications, due to the different nature of zeroes and singularities of regular functions. 
%Examples of these results, including those on power and Laurent series expansions, can be found in  \cite{saba}, \cite{runge}, \cite{appl.aperta}, \cite{powerseries}, \cite{4}, \cite{shapiro}, \cite{caterina}, \cite{singularities}. Recent results of geometric theory of regular functions appear in \cite{BohrDGS}, \cite{BlochDGS}, \cite{LTPreprint}.
%\vskip 0.3cm
Examples of this latter kind of results include those on power and Laurent series expansions, and can be found in the monograph \cite{libroGSS}. Some recent results of geometric theory of regular functions, not included in this monograph, appear in  \cite{runge}, \cite{BlochDGS}, \cite{LTPreprint}.
 \vskip 0.3cm

In this paper we study the ideals in the (non-commutative) ring of regular functions, and we prove an analogue of a classical result for one (and several) complex variables, namely the fact that if a family of holomorphic functions has no common zeroes, then it generates the entire ring of holomorphic functions. In her doctoral dissertation \cite{sarfatti}, the author proved that this was the case for regular functions as well (in fact, she showed that this was true for bounded regular functions, an analogue of the corona theorem), under the strong hypothesis that not only the functions could not have common zeroes, but also that the functions could not have zeroes on the same spheres.

Here we show that such a request is not necessary, at least for the case of regular functions (we do not consider the bounded case), by employing some delicate local properties of such functions. We show how to reduce the study of the problem to the case of holomorphic functions, and we then use the coherence of the sheaf of holomorphic functions to show that the local solution to the problem extends to a global one.

As for the state of the art in the study of the corona problem in different contexts, we refer the reader to the recent, and rather exhaustive, volume \cite{fields}, whose first chapter presents a short history of the problem itself. The papers \cite{struppa1}, \cite{treil2}, \cite{treil1}, contain significative descriptions of ideals of holomorphic functions, in connection with their maximality.  Sheaves of slice regular functions are introduced in \cite{sheaves}.

%\emph{Acknowledgments}. The first two authors express their gratitude to Chapman University, where a portion of this work was carried out.
%The class of regular functions includes convergent power series and polynomials of a quaternionic variable, whose real (and imaginary parts) are not in general harmonic as functions of four real variables. Thus, the properties of the real part of a regular function cannot be modeled upon those known for real parts of holomorphic functions.
%
%In this paper, we will show that while these functions are not harmonic, they still satisfy a version of the maximum principle, and a peculiar form of the identity
%principle, which will state that if the real part of a regular function vanishes on three suitable complex planes containing the real axis in $\hh$, then they vanish everywhere (see section $2$ for the precise statement). We will also provide a full characterization of functions on $\hh$ which are real parts of regular functions.

\section{Preliminary Results}

Let $\hh$ denote the non commutative real algebra of quaternions with standard basis $\{1,i,j,k\}$. The elements of the basis satisfy the multiplication rules
\[i^2=j^2=k^2=-1,\; ij=k=-ji,\; jk=i=-kj,\; ki=j=-ik,\]
which, if we set $1$ as the neutral element, extend by distributivity to all $q= x_0 +x_1i+x_2j+x_3k$ in $\hh$.
Every element of this form is composed by the {\em real} part $\RRe(q)=x_0$ and the {\em imaginary} part  $\IIm(q)=x_1i +x_2j +x_3k$. The {\em conjugate} of $q\in\hh$ is then $\bar{q}=\RRe(q)-\IIm(q)$ and its {\em modulus} is defined as $|q|^2=q\bar{q}.$ We can therefore calculate the multiplicative inverse of each $q\neq 0$ as $q^{-1}=\frac{\bar{q}}{|q|^2}$.
Notice that for all non real $q\in \HH$, the quantity $\frac{\IIm(q)}{|\IIm(q)|}$ is an imaginary unit, that is a quaternion whose square equals $-1$. Then we can express every $q\in \HH$ as $q=x+yI$, where $x,y$ are real (if $q\in\rr$, then $y=0$) and $I$ is an element of the unit $2$-dimensional sphere of purely imaginary quaternions,
$$\mathbb{S}=\{q\in \HH \ | \ q^{2}=-1 \}.$$
In the sequel, for every $I\in \mathbb{S}$ we will denote by $L_I$ the plane $\mathbb{R}+\mathbb{R}I$, isomorphic to $\mathbb{C}$ and, if $\Omega$ is a subset of $\HH$, by $\Omega_I$ the intersection $\Omega \cap L_I$.
As explained in \cite{libroGSS}, the natural domains of definition for slice regular functions are the symmetric slice domains. These domains actually play the role played by domains of holomorphy in the complex case:
\begin{definition}\label{slicedomain}\index{slice domain}
Let $\Omega$ be a domain in $\hh$ that intersects the real axis. Then:
\begin{enumerate}
\item $\Omega$ is called a \emph{slice domain} if, for all $I \in \s$, the intersection $\Omega_I$ with the complex plane $L_I$ is a domain of $L_I$;
\item $\Omega$ is called a \emph{symmetric domain} if for all $x,y\in \rr$, $x+yI\in \Omega$ implies $x+y\s\subset \Omega$.
\end{enumerate}
\end{definition}
We can now recall the definition of slice regularity. From now on, $\Omega$ will always be a symmetric slice domain in $\hh$, unless differently stated.
\begin{defn}\label{holomorphic}
A function $f: \Omega \rightarrow \HH$ is said to be {\em (slice) regular} if, for all $I\in \mathbb{S}$,  its restriction $f_I$ to $\Omega_I$ has continuous partial derivatives and is \emph{holomorphic}, i.e., it satisfies
$$\overline{\partial}_I f(x+yI):=\frac{1}{2}\Big(\frac{\partial}{\partial x}  +I\frac{\partial}{\partial y}\Big)f_I(x+yI)=0$$
for all $x+yI \in \Omega_I$.
\end{defn}
\noindent
%In the sequel we may refer to the vanishing of $\overline{\partial}_I f$ saying that the restriction $f_I$ is \emph{holomorphic} on $\Omega_I$.

%The preliminary results that follow will be stated for regular functions defined on open balls centred at the origin, 
%$B=B(0,R)=\{q \in \HH \, |\, |q|<R \}$, even if  in most of the
%cases these results hold, with appropriate changes,   for regular functions defined on slice domains, \cite{ext}.
%\noindent From a slicewise version of the Cauchy Integral Formula, see \cite{GSAdvances}, it follows that the $n$-th coefficient $a_n$ of the power series expansion of a regular function $f:B(0,R) \rightarrow \HH$ has an integral representation. More precisely, for all $n \geq 0$, if $I \in \s$, $r \in (0,R)$ and $\Delta_I(0,r)=\{ z \in L_I \ | \ |z|<r\} $, then
%\begin{equation}\label{intrep}
%a_n= \frac{1}{2\pi I}\int_{\partial \Delta_I(0,r)} \frac{f(z)}{z^{n+1}}dz.
%\end{equation}
\noindent A basic result in the theory of regular functions, that relates slice regularity and classical holomorphy, is the following, \cite{libroGSS, 3}:
\begin{lem}[Splitting Lemma]\label{split}
If $f$ is a regular function on $\Omega$, then for every $I \in \mathbb{S}$ and for every $J \in \mathbb{S}$, $J$ orthogonal to $I$, there exist two holomorphic functions $F,G:\Omega_I \rightarrow L_I$, such that for every $z=x+yI \in \Omega_I$, it holds
$$f_I(z)=F(z)+G(z)J.$$
\end{lem}
\noindent One of the first consequences of the previous result is the following version of the Identity Principle, \cite{3}:
\begin{theorem}[Identity Principle]\label{Id}
Let $f$ be a regular function on $\Omega$. Denote by $Z_f$ the zero set of $f$, $Z_f=\{ q \in \Omega | \, f(q)=0 \}$. If there exists $I \in \mathbb{S}$ such that $\Omega_I \cap  Z_f$ has an accumulation point in $\Omega_I$, then
$f$ vanishes identically on $\Omega$.
\end{theorem}

\noindent In the sequel we will use an important extension result (see \cite{saba}, \cite{advances2}) that we will present in the following special formulation:

%\begin{lemma}[Extension Lemma] \label{extensionlemma}\index{extension $\ext(f_I)$}
%Let $\Omega$ be a symmetric slice domain and let $I \in \s$. If $f_I : \Omega_I \to \hh$ is holomorphic then there exists a unique regular function $g : \Omega \to \hh$ such that $g_I = f_I$ in $\Omega_I$. The function $g$ will be denoted by $\ext(f_I)$.
%\end{lemma}
%
%
%
%
%%\begin{oss}\label{ext1}
%%If $f_I$ is a holomorphic function on a disc $B_I=B(0,R)\cap L_I$ and its power series expansion is
%%\[ f_I(z)=\sum_{n=0}^{\infty}z^na_n, \quad \text{ with $\{a_n\}_{n\in \mathbb{N}}\subset \HH$}, \]
%%then the unique regular extension of $f_I$ to the whole ball $B(0,R)$ is the function defined as
%%\[ \ext (f_I) (q) = \sum_{n=0}^{\infty}q^na_n .\]
%%The uniqueness is guaranteed by the Identity Principle \ref{Id}.
%%\end{oss}

%On open balls centred at the origin the $*$-product of two regular functions can be defined by means of power series, \cite{zeri}.

%The generalization to the symmetric slice domains is based on the followin result, which proof is in \cite{ext}.
\begin{lem}[Extension Lemma]\label{extensionlemma}
Let $\Omega$ be a symmetric slice domain and choose $I \in \s$. If $f_I: \Omega_I \rightarrow \HH$ is holomorphic, then setting
\begin{equation*}
f(x+yJ)=\frac{1}{2}[f_I(x+yI)+f_I(x-yI)] +J\frac{I}{2}[f_I(x-yI)-f_I(x+yI)]
\end{equation*}
extends $f_I$ to a regular function $f: \Omega \rightarrow \HH$. Moreover $f$ is the unique extension and it is denoted by $\ext(f_I)$.
\end{lem}

\noindent The product of two regular functions is not, in general, regular. To guarantee the regularity we have to use a different  multiplication operation, the $*$-product. From now on, if $F$ is a holomorphic function, we will use the notation:
$$
\hat F(z) := \overline{F(\bar z)}.
$$

%\noindent In order to define the regular product of $f$ and $g$, regular functions on a symmetric slice domain $\Omega$, take $I,J \in \s$, with $I$ orthogonal to $J$, and choose holomorphic functions $F,G,H,K: \Omega_I \rightarrow L_I$ such that for all $z\in \Omega_I$
%$$f_I(z)=F(z)+G(z)J \quad \text{and} \quad g_I(z)=H(z)+K(z)J.$$
%Let $f_I * g_I : \Omega_I \rightarrow L_I$ be the holomorphic function defined as
%$$f_I * g_I(z)=[F(z)H(z)-G(z)\overline{K(\overline{z})}]+[F(z)K(z)+G(z)\overline{H(\overline{z})}]J.$$
%The following definition is given in \cite{ext}.

\begin{definition}\label{R-starprodotto}\index{regular product} \index{$*$-product}
Let $f,g$ be regular functions on a symmetric slice domain $\Omega$. Choose $I,J \in \s$ with $I \perp J$ and let $F,G,H,K$ be holomorphic functions from $\Omega_I$ to $L_I$ such that $f_I = F+GJ, g_I = H+KJ$. Consider the holomorphic function defined on $\Omega_I$ by
\begin{equation}\label{prodottostar}
f_I*g_I(z) = \left[F(z)H(z)-G(z)\hat K(z)\right] + \left[F(z)K(z)+G(z)\hat H(z ) \right]J.
\end{equation}
Its regular extension $\ext(f_I*g_I)$ is called the \emph{regular product} (or \emph{$*$-product}) of $f$ and $g$ and it is denoted by $f*g$.
\end{definition}

\noindent Notice that the $*$-product is associative but generally is not commutative. Its connection with the usual pointwise product is stated by the following result.
\begin{pro}\label{trasf}
Let $f(q)$ and $g(q)$ be regular functions on $\Omega$. Then, for all $q\in \Omega$,
\begin{equation}%\label{prodstar}
f*g(q)= \left\{ \begin{array}{ll}
 f(q)g(f(q)^{-1}qf(q)) & \text{if} \quad f(q)\neq 0\\
0 & \text{if} \quad f(q)=0
\end{array}
\right.
\end{equation}
\end{pro}

\begin{corollary}\label{productzeros}
If $f,g$ are regular functions on a symmetric slice domain $\Omega$ and $q \in \Omega$, then $f*g(q) = 0$ if and only if $f(q) = 0$ or $f(q) \neq 0$ and $g(f(q)^{-1} q f(q))=0$.
\end{corollary}

\noindent To illustrate the natural meaning of the $*$-product of two regular functions,  we consider two quaternionic power series, $\sum_{n=0}^\infty q^na_n$ and $\sum_{n=0}^\infty q^nb_n$, both centered at zero and with radius of convergence $R>0$. These power series define two regular functions $f(q)=\sum_{n=0}^\infty q^na_n$ and $g(q)=\sum_{n=0}^\infty q^nb_n$ on the open ball $B(0,R)\subseteq \HH$ centered at $0$ and with radius $R$ (see e.g. \cite{libroGSS}). Now, (polynomials and) power series with coefficients in a non commutative ring are classically endowed with the \emph{Cauchy product}, that even in the non commutative case is still defined as
\begin{equation}\label{coeff}
(\sum_{n=0}^\infty q^na_n)\cdot (\sum_{n=0}^\infty q^nb_n)= \sum_{n=0}^\infty q^nc_n  \textnormal{\ \ \ with\ \ \ } c_n=\sum_{m=0}^n a_mb_{n-m}
\end{equation}
 so that the sequence of coefficients $\{c_n\}$ is obtained by the convolution of the sequences $\{a_n\}$ and $\{b_n\}$. 
It turns out that

\begin{pro}
The $*$-product of the regular functions $f(q)$ and $g(q)$ coincides with the Cauchy product of their power series expansions, i.e.
\[
f*g(q)=(\sum_{n=0}^\infty q^na_n)\cdot (\sum_{n=0}^\infty q^nb_n),  
\]
on $B(0,R)$.
\end{pro}
\begin{proof}
Let us consider the coefficients $\{a_n\}, \{b_n\}, \{c_n\}$ of the power series appearing in Equation \eqref{coeff}. Choose $I,J$ in $\s$ with $I\perp J$ and write
\[
a_n=\alpha_n+\beta_nJ   \textnormal{\ \ \ \ and \ \ \ \ } b_n=\gamma_n +\delta _n J
\]
for suitable $\alpha_n, \beta_n, \gamma_n, \delta_n$ in $L_I$. A direct computation shows that the splitting of $c_n$ is 
\[
c_n=\sum_{m=0}^n(\alpha_m\gamma_{n-m}-\beta_m\bar \delta_{n-m}) + \sum_{m=0}^n(\alpha_m\delta_{n-m}+\beta_m\bar \gamma_{n-m})J
\]
and a comparison with equation \eqref{prodottostar} leads to the conclusion of the proof.
\end{proof}

\noindent It is immediate, and useful for the sequel, to notice that if $\{a_n\}$  are all real numbers, then we have
\[
f*g(q)=(\sum_{n=0}^\infty q^na_n)\cdot (\sum_{n=0}^\infty q^nb_n)=fg(q)=gf(q)=(\sum_{n=0}^\infty q^nb_n)\cdot (\sum_{n=0}^\infty q^na_n)=g*f(q).
\]
on the whole domain of convergence $B(0,R)$ of the power series, i.e. the $*$-product and the pointwise product coincide (and are commutative). Hence power series with real coefficients define, on their domains of convergence, regular functions that behave nicely with respect to the $*$-product; these functions are called  \emph{slice preserving} regular functions, since, for all $I\in \s$, they map subsets of $L_I$ into $L_I$. 

The following operations are naturally defined in order to study the zero set of regular functions.

\begin{definition}\label{R-coniugata}\index{regular conjugate}
Let $f$ be a regular function on a symmetric slice domain $\Omega$. Choose $I,J \in \s$ with $I \perp J$ and let $F,G$ be holomorphic functions from $\Omega_I$ to $L_I$ such that $f_I = F+GJ$. If $f_I^c$ is the holomorphic function defined on $\Omega_I$ by
\begin{equation}
f_I^c(z) = \hat F(z) - G(z)J.
\end{equation}
Then the \emph{regular conjugate} of $f$ is the regular function defined on $\Omega$ by $f^c=\ext(f_I^c)$, and the
 \emph{symmetrization} of $f$ is the regular function defined on $\Omega$ by $f^s = f*f^c = f^c*f$.
%Moreover the regular reciprocal $f^{-*}$  of   $f$ is defined as $f^{-*}=f^{-s}f^c$ on $\Omega\setminus Z_f$
\end{definition}
\noindent If the regular function $f:\Omega \to \hh$ is such that
 $
 f_I(z)=F(z)+G(z)J,
 $
with
$F,G: \Omega_I \to L_I$  holomorphic functions, then
it easy to see that  (see, e.g., \cite{libroGSS})
\begin{equation}
f^s_I=f_I*f^c_I=f^c_I*f_I
=F(z)\hat F( z)+G(z)\hat G( z).
\end{equation}
Hence $f^s(\Omega_I)\subseteq L_I$ for every $I\in \s$, i.e.,  $f^s$ is slice preserving. Moreover
if  $g$ is a  regular function on $\Omega$, then
\begin{equation}\label{uuu}
(f*g)^c=g^c*f^c \quad\textnormal{and} \quad (f*g)^s=f^sg^s=g^sf^s.
\end{equation}

\noindent Zeroes of regular functions have a nice geometric property:

\begin{theorem}
Let $f$ be a regular function on a symmetric slice domain $\Omega$. If $f$ does not vanish identically, then its zero set consists of isolated points or isolated $2$-spheres of the form $x +y \mathbb{S}$ with $x,y \in \mathbb{R}$, $y \neq 0$.
\end{theorem}

%
%\begin{proposition}\label{R-fessegesse3} Let $f$  be a regular function on a symmetric slice domain $\Omega$, such that $f(\Omega_I)\subseteq L_I$ for all $I\in \s$. Then $f^c(q)=f(q)$ and $f^s(q) = f(q)^2$ for all $q \in \Omega$.
%\end{proposition}

\noindent Notice that $f(q)^{-1} q f(q)$ belongs to the same sphere $x+y\s$ as $q$. Hence each zero of $f*g$ in $x+y\s$ corresponds to a zero of $f$ or to a zero of $g$ in the same sphere.
%However, the correspondence between $Z_{f*g}$ and $Z_{f}\cup Z_{g}$ need not be one-to-one:
\begin{lemma}\label{symmetrizationreal}
Let $f$ be a regular function on a symmetric slice domain $\Omega$ and let $f^s$ be its symmetrization. Then for each $S=x+y\s \subset \Omega$ either $f^s$ vanishes identically on $S$ or it has no zeroes in $S$.
\end{lemma}

\noindent The \emph{regular reciprocal} $f^{-*}$ of a regular function $f$ defined on a symmetric slice domain $\Omega$  can now be defined in $\Omega\setminus Z_{f^s}$ as
\begin{equation}
f^{-*}=(f^{s})^{-1}f^c,
\end{equation}
where $Z_{f^s}$ denotes the zero set of the symmetrization $f^s$.
\begin{oss}\label{oss}{\rm
If $f$ is a regular function defined on a slice symmetric domain of $\hh$, then its regular reciprocal  $f^{-*}=(f^{s})^{-1}f^c$ has a sphere of poles at $Z_{f^s}$ and is a quasi regular function in the sense of \cite{singularities}.}
\end{oss}

\section{Ideals generated by two regular functions}

In this section we will prove that if $f_1$ and $f_2$ are two regular functions with no common zeroes on a symmetric  slice domain $\Omega$, then they generate the entire ring of regular functions on $\Omega$, i.e. there are two regular functions $h_1$ and $h_2$ on $\Omega$ such that $f_1*h_1+f_2*h_2=1$.
\vskip 0.3cm

We begin by proving a local version of this result for holomorphic functions (in the sense of Definition \ref{holomorphic}), following the approach used in the case of several complex variables. 
%that two regular functions having no common zeroes, restricted to each slice locally generate the ring of regular functions.

\begin{theorem} \label{teolocale}
%Let $q=x+yL\in\hh$ and 
Let $f_1, f_2$ be two functions, regular in a symmetric slice domain $\Omega$ without common zeroes. 
%$q$ and not simultaneously vanishing at $q$. 
Then, for any $I\in \s$,
% it is possible to find  $I\in \s$ 
%and a symmetric domain $W$ in $\Omega_I$ containing $x+yI$ and $x-yI$, 
the equation
\begin{equation}\label{localedue}
f_1*h_1+f_2*h_2=1.
\end{equation}
restricted to $\Omega_I$ 
has local holomorphic solutions $h_1, h_2$ at any point of $\Omega_I$.
\end{theorem}
\begin{proof}
%Let $q=x+yL\in \Omega$. Notice that, by hypothesis, $x+y\s$ cannot be a spherical zero for both $f_1$ and $f_2$.
% have an isolated  non common zero in $x+y\s$. %Since isolated zeroes of regular functions are a countable set, there exists $I\in\s$ such that $f_1,f_2,f_1^c,f_2^c$ do not vanish in a neighborhood of $x+y\s \cap L_I$.
By the Splitting Lemma, for any $I\in\s$, we can represent, for $\ell=1,2,$ the functions $f_\ell$ via functions holomorphic in $\Omega_I$ as
%a domain in $L_I$ containing $(x+y\s) \cap L_I$ as
$$
f_\ell(z)=f_{\ell|I}(z)=F_\ell(z)+G_\ell(z)J,
$$
where $J\in \s$ is orthogonal to $I$.
Similarly, the functions $h_\ell$ that we are looking for can be written as
$$
h_\ell(z)=h_{\ell|I}(z)=H_\ell(z)+K_\ell(z)J,
$$
for suitable holomorphic functions $H_\ell$ and $K_\ell$.
%Let us apply Lemma \ref{borsuk} to $f_1$, and choose $I$ such that, without loss of generality, $F_1(x+yI)F_1(x-yI)\neq 0$. As a consequence we get that also $\overline{F_1(x+yI)}\,\overline{F_1(x-yI)}\neq 0$,
%$f_1(x+yI)f_1(x-yI)\neq 0$ and $f_1^c(x+yI)f_1^c(x-yI)\neq 0$. Notice that the previous inequalities hold outside the discrete subset of $\Omega_I$ consisting of all zeroes of $F_1(z)$ and $\hat F_1(z)$, hence in a symmetric domain $U$ in $\Omega_I$ containing both $x+yI$ and $x-yI$.
Using \eqref{prodottostar}, it is immediate to see that \eqref{localedue} can be rewritten as a system of two equations for holomorphic functions in $L_I$, namely, omitting the variable $z$,
%\begin{equation}\label{sistemalocaledue}
%\left\{\begin{array}{l}
%F_1(z)H_1(z)-G_1(z)\overline{K_1({\bar z})}+F_2(z)H_2(z)-G_2(z)\overline{K_2({\bar z})}=1\\
%F_1(z)K_1(z)+G_1(z)\overline{H_1({\bar z})}+F_2(z)K_2(z)+G_2(z)\overline{H_2({\bar z})}=0.\end{array}\right.
%\end{equation}
%For simplicity of notation, in the future if $F(z)$ is a holomorphic function, we will write $\hat{F}(z):=\overline{F({\bar z})},$ so that \eqref{sistemalocaledue} can actually be rewritten as
\begin{equation}\label{sistemalocaledueconcappelli}
\left\{\begin{array}{l}
F_1H_1-G_1\hat K_1+F_2H_2-G_2\hat K_2=1\\
F_1K_1+G_1\hat H_1+F_2K_2+G_2\hat H_2=0.\end{array}\right.
\end{equation}
%The hypothesis that $f_1$ and $f_2$ do not simultaneously vanish in $q$ implies that the functions $F_1, G_1, F_2, G_2$ also do not simultaneously vanish in $q$. Therefore,
Since $f_1$ and $f_2$ do not have common zeroes in $\Omega_I \subset \Omega$, the same holds true for $F_1,G_1,F_2,G_2$.
%Since $F_1$ does not vanish on $U$, there exist 
Hence, a classical one complex variable result implies that there exist $H_1, K_1, H_2, K_2$, holomorphic in $\Omega_I$, which define a solution of the first equation of \eqref{sistemalocaledueconcappelli}.
%The existence of such a domain is guaranteed because the zeroes of $f_1$ and $f_2$ (and their conjugates) in $L_I$ form a discrete set.
In general, the functions $H_1, K_1, H_2, K_2$ will not define a solution of system \eqref{sistemalocaledueconcappelli}. However, one can modify the solution to the first equation by adding an element of the syzygies of $(F_1,G_1,F_2,G_2)$ and try to solve the system. Since the latter functions have no common zeroes on $\Omega_I$, their syzygies (see, e.g., \cite{libro daniele}) are generated by the columns of the following matrix
\begin{equation*}
A={\begin{pmatrix}
                      G_1 & F_2 & G_2 & 0 & 0 & 0 \\
                      -F_1 & 0 & 0 & F_2 & G_2 & 0 \\
                      0 & -F_1 & 0 & -G_1 & 0 & G_2 \\
                      0 & 0 & -F_1 & 0 & -G_1 & -F_2 \\
                    \end{pmatrix}.
}
\end{equation*}
Hence the general solution to the first equation of \eqref{sistemalocaledueconcappelli} is given by
\begin{equation}
\left(
\begin{array}{l}
H_1+\hat\beta_1G_1+\hat\beta_2F_2+\hat\beta_3G_2\\
-\hat K_1-\hat\beta_1F_1+\hat\beta_4F_2+\hat\beta_5G_2\\
H_2-\hat\beta_2F_1-\hat\beta_4G_1+\hat\beta_6G_2\\
-\hat K_2 -\hat\beta_3F_1-\hat\beta_5G_1-\hat\beta_6F_2
\end{array}
\right)
\end{equation}
 where $\beta_1,\ldots, \beta_6$ are arbitrary holomorphic functions in $\Omega_I$.
%which in compact form can be written as
%\begin{equation}
% H+A  \hat\beta,
%\end{equation}
%with $H= \ ^t(H_1,-\hat K_1, H_2,-\hat K_2)$.
Consider now the matrix $B$ of holomorphic functions defined by
\begin{equation}
B={\begin{pmatrix}
                      \hat F_1 & 0 & 0 & -\hat F_2 & -\hat G_2 & 0 \\
                      \hat G_1 & \hat F_2 & \hat G_2 & 0 & 0 & 0 \\
                      0 & 0 & \hat F_1 & 0 & \hat G_1 & \hat F_2 \\
                      0 & -\hat F_1 & 0 & -\hat G_1 & 0 & \hat G_2 \\
                    \end{pmatrix}.
}
\end{equation}
In order
%for $H+A\hat\beta$
to obtain a solution of \eqref{sistemalocaledueconcappelli} we now need to request that the vector
\begin{equation}\label{syzygies}
\left(
\begin{array}{l}
K_1+\beta_1\hat F_1-\beta_4\hat F_2-\beta_5\hat G_2\\
\hat H_1+\beta_1\hat G_1+\beta_2 \hat F_2+\beta_3\hat G_2\\
K_2 +\beta_3\hat F_1+\beta_5\hat G_1+\beta_6\hat F_2\\
\hat H_2-\beta_2\hat F_1-\beta_4\hat G_1+\beta_6\hat G_2\\
\end{array}
\right)
\end{equation}
belongs to the syzygies of $(F_1,G_1,F_2,G_2)$. That is,  setting $H=\ ^t(K_1,\hat H_1, K_2, \hat H_2)$, we need to find $\beta=\ ^t(\beta_1,\ldots, \beta_6)$ and $\alpha=\ ^t(\alpha_1,\ldots, \alpha_6)$ vectors of holomorphic functions such that
\[H+B\beta=A\alpha,\]
namely such that
\begin{equation}\label{capelli}
\big(\begin{array}{c  c}
 A  , & -B
 \end{array}
 \big)
 \left(\begin{array}{l}
 \alpha \\ \beta
 \end{array}
  \right)
   =H.
 \end{equation}
Our next goal is to establish that the rank of the $(4\times 12)$-matrix $(A, -B)$ equals $4$ on the entire $\Omega_I$. 
%Recalling that $F_1$ and $\hat F_1$ do not vanish in $U$, 
Since $F_1,G_1,F_2,G_2$ have no common zeroes, it is easy to prove that both $A$ and $B$ have rank $3$ at each point $z\in \Omega_I$. Consider for instance $A$ and denote by $A^1,\ldots, A^6$ its columns. 
If $F_1(z)\neq 0$, then $\{A^1,A^2,A^3\}$ is a maximal subset of linearly independent columns  on a neighborhood of $z$. If $F_1(z)=0$ and $G_1(z)\neq 0$, we can take as a maximal subset of linearly independent columns $\{A^1,A^4,A^5\}$. If both $F_1(z)$ and $G_1(z)$ vanish, we proceed analogously considering $F_2$ and $G_2$.  
%in both matrices the first three columns are a maximal subset of linearly independent columns on $U$.
%Denote by $A^1,\ldots, A^6$ the columns of $A$ and by $B^1,\ldots,B^6$ the columns of $B$. 
The rank of $(A, -B)$ is not maximum at a point $z\in \Omega_I$ if and only if all columns of $B$ are linear combinations of columns of $A$, which is equivalent to the condition that all columns of $B$ belong to the syzygies of $(F_1,G_1,F_2,G_2)$.
%all the determinants of the six $(4 \times 4)$-matrices
%\[ M_1=(A^1,A^2,A^3,B^1), \ \ldots \ ,\  M_6=(A^1,A^2,A^3,B^6)\]
% vanish at $z$. 
 Hence the rank of $(A,-B)$ is $3$ where (in $\Omega_I$) the following
 six equations are simultaneously satisfied:
%\begin{equation}\label{sistema}
%\left\{
%\begin{array}{l}
% F_1^2(F_1\hat F_1+G_1\hat G_1 )=0\\
% F_1^2(F_1\hat F_2+G_2\hat G_1 )=0\\
% F_1^2(F_1\hat G_2-F_2\hat G_1 )=0\\
% F_1^2(G_1\hat F_2-G_2\hat F_1 )=0\\
% F_1^2(F_2\hat F_1+G_1\hat G_2 )=0\\
% F_1^2(F_2\hat F_2+G_2\hat G_2 )=0\\
%\end{array}
%\right.
%%
%\end{equation}
%%composed by the first three columns of $A$ and one column of $-B$
%Taking into account that $F_1$ is nonvanishing in $U$, the rank equals $3$ if and only if the following
%six equations are contemporarily satisfied:
\begin{align}
&F_1\hat F_1+G_1\hat G_1=0  \label{1}\\
&F_1\hat F_2+G_2\hat G_1 =0 \label{2}\\
&F_1\hat G_2-F_2\hat G_1 =0 \label{3}\\
&G_1\hat F_2-G_2\hat F_1 =0 \label{4}\\
&F_2\hat F_1+G_1\hat G_2 =0 \label{5}\\
&F_2\hat F_2+G_2\hat G_2 =0 \label{6}
\end{align}
Equations \eqref{1} and \eqref{6} can be written in $\Omega_I$  as the quaternionic  equations
%\begin{equation}
%\left\{
%\begin{array}{l}
$f_1^s(z)=0$ and
$f_2^s(z)=0$.
%\end{array}
%\right.
%
%\end{equation}
We will now investigate the meaning of the other terms. Using \eqref{prodottostar} and the fact that $\Omega_I$ is symmetric
(i.e. if it contains $z$ then it contains $\bar z$ as well), we get
\[(f_1^c*f_2)_I(z)=(F_2(z)\hat F_1(z)+G_1(z)\hat G_2(z))-(G_1(z)\hat F_2(z)-G_2(z)\hat F_1(z))J\]
\[(f_2^c*f_1)_I(z)=(F_1(z)\hat F_2(z)+G_2(z)\hat G_1(z))+(G_1(z)\hat F_2(z)-G_2(z)\hat F_1(z))J\]
\[(f_1^c*f_2)_I(\bar z)=\overline{(F_1(z)\hat F_2(z)+G_2(z)\hat G_1(z))}+\overline{(F_1(z)\hat G_2(z)-F_2(z)\hat G_1(z))}J\]
\[(f_2^c*f_1)_I(\bar z)=\overline{(F_2(z)\hat F_1(z)+G_1(z)\hat G_2(z))}+(F_1(z)\hat G_2(z)-F_2(z)\hat G_1(z))J.\]
Hence if the matrix $(A,-B)$ has rank $3$ at $z\in \Omega_I$, then equations \eqref{2}-\eqref{5} imply that $(f_1^c*f_2)_I(z)=(f_2^c*f_1)_I(z)=(f_1^c*f_2)_I(\bar z)=(f_2^c*f_1)_I(\bar z)=0$. Consequently if $(A, -B)$ has rank $3$ at $z\in U$, then we have
%\begin{eqnarray}
%\label{sistemacompatto}
%\left\{
 \begin{align}
 f_1^s(z)=0 \label{i}\\
 f_1^c*f_2(z)=0 \label{ii}\\
 f_2^c*f_1(z)=0 \label{iii}\\
 f_1^c*f_2(\bar z)=0 \label{iv}\\
 f_2^c*f_1(\bar z)=0\label{v}\\
f_2^s(z)=0 \label{vi}
 \end{align}
% \right.
% \end{equation}
Let $z=x+yI$. From equations \eqref{i} and \eqref{vi} we obtain that both $f_1$ and $f_2$ have a (non common and hence non spherical) zero in the sphere $x+y\s$. 
%Suppose first that $f_1(z)\neq 0$. 
Equation \eqref{i} can be written as
\[
%f_1^c(z)f_1((f_1^c(z))^{-1}zf_1^c(z))=0.
f_1^c*f_1(z)=0
\]
which, by Proposition \ref{trasf} leads to two possibilities: %(1) $f_1^c(z)=0$, or (2) $f_1^c(z)\neq 0$ and $f_1((f_1^c(z))^{-1}zf_1^c(z))=0$.
\begin{enumerate}
\item[(a)]  $f_1^c(z)=0$ or
\item[(b)] $f_1^c(z)\neq 0$ and $f_1((f_1^c(z))^{-1}zf_1^c(z))=0$.
\end{enumerate}
In case (a), we have that $f_1^c(\bar z)\neq 0$, since $x+y\s$ is not a spherical zero of ($f_1$ and hence of) $f_1^c$. Thanks to  Proposition \ref{trasf}, equation \eqref{iv} becomes
\[f_1^c(\bar z)f_2((f_1^c(\bar z))^{-1}\bar z f_1^c(\bar z))=0,\]
which implies that 
\begin{equation}\label{primozero}
f_2((f_1^c(\bar z))^{-1}\bar z f_1^c(\bar z))=0.
\end{equation} 
Moreover \eqref{i} yields that $x+y\s$ is a spherical zero of $f_1^s$, and hence that
%$f_1^s(\bar z)=0$ and hence that
\[0=f_1^s(\bar z)=f_1^c(\bar z)f_1((f_1^c(\bar z))^{-1}\bar z f_1^c(\bar z)),\]
%Since $x+y\s$ is not a spherical zero of $f_1$, $f_1^c(\bar z)\neq 0$ and 
leading to 
\begin{equation}\label{secondozero}
f_1((f_1^c(\bar z))^{-1}\bar z f_1^c(\bar z))=0.
\end{equation} 
The hypothesis that $f_1$ and $f_2$ have no common zeroes together with \eqref{primozero} and \eqref{secondozero} gives us a contradiction.\\
In  case (b), again thanks to Proposition \ref{trasf}, equation \eqref{ii} 
$$
f_1^c(z)f_2((f_1^c(z))^{-1}zf_1^c(z))=0
$$
yields that $f_2$ vanishes at $(f_1^c(z))^{-1}zf_1^c(z)$ which is a zero of $f_1$. Again a contradiction.
In conclusion, equations \eqref{i}-\eqref{vi} (and hence equations \eqref{1}-\eqref{6}) are never simultaneously satisfied, that implies that the matrix $(A, -B)$ has rank $4$ at all points of $\Omega_I$. 
%at first the case in which $f_1$  has a zero in $x+y\s$. If $q=x+yL$, let $Z=\{ x_n+y_n\s \}_{n\in\nn} \subset \Omega$ be the set of spheres, different from $x+y\s$, containing zeroes of $f_1^c*f_2$. Then
%$V=U \setminus Z$ is a symmetric domain in $L_I$ containing $x+yI$ (and $x-yI$) on which $f_1^c*f_2$ never vanishes.
%, with the only possible exception of $f(x-yI)=0$.
%To prove this fact,  %consider at first the case in which $f_1$  has a zero in $x+y\s$.
%thanks to Lemma \ref{symmetrizationreal} and to \eqref{trasf}, we obtain
%\[0=f^s(x+yI)= f_1^c*f_1(x+yI)=f_1^c(x+yI)f_1(f_1^c(x+yI)^{-1}(x+yI)f_1^c(x+yI)).\]
%As we pointed out  $f_1^c(x+yI)\neq 0$, and therefore
%\[f_1(f_1^c(x+yI)^{-1}(x+yI)f_1^c(x+yI))=0.\]
%Since, by hypothesis,  $f_2$ cannot vanish at a zero of $f_1$, we get
%\[f_1^c*f_2(x+yI)= f_1^c(x+yI)f_2(f_1^c(x+yI)^{-1}(x+yI)f_1^c(x+yI))\neq 0.\]
%An analogous computation shows that $f_1^c*f_2(x-yI)\neq 0$, thus proving our assertion concerning $V$. As a consequence
%one of the holomorphic determinants appearing in \eqref{sistema} (namely one of those related to \eqref{4} or \eqref{5})
%cannot vanish at $x+yI$ and one cannot vanish at $x-yI$: let us call this functions $N_1$ and $N_2$.  Let $Z_N\subset \Omega_I$ be the set of common zeroes of  $N_1$ and $N_2$. Then $W=V \setminus Z_N$ is a symmetric domain in $\Omega_I$ containing (and hence a neighborhood of) $x+yI$ and $x-yI$. Since for every point  $p$ of $W$ one of the holomorphic determinants $N_1(z), N_2(z)$  is non vanishing at $p\in W$, 
Therefore, using the classical Rouch\'e - Capelli method it is now possible to find a local holomorphic solution $(\alpha, \beta)$
%$H=\ ^t(K_1,\hat H_1, K_2, \hat H_2)$ 
of system \eqref{capelli} in the neighborhood of each point $z\in \Omega_I$. This gives us a local holomorphic solution of system \eqref{sistemalocaledueconcappelli} and hence of equation \eqref{localedue}.
%a neighborhood $A_p\subseteq W\subseteq \Omega_I$ of $p$.  %on the entire domain $W\subseteq L_I$.
%This solution is such that the holomorphic functions $h_1(z)=H_1(z)+K_1(z)J$ and $h_2(z)=H_2(z)+K_2(z)J$ are local holomorphic solutions of the restriction of equation \eqref{localedue} to $A_p$, i.e., it is such that
%$$
%f_1*h_1(z)+f_2*h_2(z) = 1
%$$
%for all $z\in A_p$. %(with $x+yI \in W$).
%%To conclude the proof we need to use the Extension Lemma \ref{extensionlemma} to extend $h_1$ and $h_2$ to the slice symmetric open neighborhood $\Sigma\subseteq \Omega$ of $q$ defined by
%%$$
%%\Sigma= \bigcup_{x+yI \in W} (x+y\s).
%%$$
%In the remaining case in which $f_1$ has no zeroes in $x+y\s$, then $f^s_1=F_1\hat F_1+G_1\hat G_1\neq 0$ at both $x+yI$ and $x-yI\in x+y\s$. Setting $N=F_1^2f^s_1= F_1^2(F_1\hat F_1+G_1\hat G_1 )$ we proceed exactly as in the previous case and conclude.
%%Recalling that $f_1^c$ never vanishes in $U$, we get that the second equation in \eqref{sistema} can be written as (see \ref{trasf})
%%\[ f_1^c*f_2(z)=f_1^c(z)f_2(f_1^c(z)^{-1}zf_1^c(z))=0,\]
%%and hance it is satisfied if and only if $f_2(f_1^c(z)^{-1}zf_1^c(z))=0$.
%%If $Z=\{ x_n+y_n\s \}_{n\in\nn} \subset \Omega$ are the spheres containing the common zeros of $f_1$ and $f_2$, then $V=U \setminus Z$ is a symmetric domain in $L_I$ containing $q$.
%%Let us restrict our attention to the domain $V\subseteq L_I$. Notice that for all $z\in V$,
%%\[0=f_1^c*f_1(z)\]
%%
%%Since for all $z\in V$, $f_1^c(z)^{-1}zf_1^c(z)$ is a zero of $f_1$
\end{proof}

To find a global solution of \eqref{localedue} on $\Omega_I$ % identifying a solution of \eqref{localedue} in the entire domain $\Omega\subseteq \hh$, 
we will apply results from the theory of  analytic sheaves. More precisely we will use the following consequence of Cartan Theorem B, see \cite{krantz}.
\begin{theorem}\label{fasci}
Let $D\subseteq \C^n$ be a pseudo convex domain, and let $(\mathcal F, D)$ be a coherent analytic sheaf. Suppose that there exist finitely many global sections $s_1,\ldots s_k \in \Gamma(D, \mathcal F)$ such
that  $(s_1)_z,\ldots, (s_k)_z$ generate the stalk ${\mathcal F}_z$ over each $z\in D$. Then for any global section $g\in \Gamma(D,\mathcal F)$, there
exist $g_1,\dots , g_k \in \Gamma(D,\mathcal O)$ holomorphic functions on $D$ such that $g= s_1g_1 + \cdots + s_kg_k$.
\end{theorem}

In our setting the sheaf $(\mathcal F, D)$ will be the coherent sheaf $(\mathcal O^4, \Omega_I)$ of $4$-tuples of germs of holomorphic functions on $\Omega_I$.

\begin{theorem}\label{teoglobale}
Let $f_1$, $f_2$ be regular functions on a symmetric slice domain $\Omega\subseteq\hh$, with no common zeroes in $\Omega$. Then there exist $h_1$ and $h_2$ regular functions  on $\Omega$ such that
\begin{equation}\label{localeduebis}
f_1*h_1+f_2*h_2=1
\end{equation}
on $\Omega$.
\end{theorem}

\begin{proof}
Fix $I\in \s$ and, with the notation of the proof of Theorem \ref{teolocale}, consider the linear system
\begin{equation}\label{capellibis}
\big(\begin{array}{c  c}
 A  , & -B
 \end{array}
 \big)
 \left(\begin{array}{l}
 \alpha \\ \beta
 \end{array}
  \right)
   =H
 \end{equation}
associated with equation \eqref{localeduebis} restricted to $\Omega_I$. 
%let $(\alpha,\beta)$ be a local holomorphic solution of system \eqref{capelli}, corresponding to equation \eqref{localeduebis}, on $\Omega_I$. 
In the language of analytic sheaves, the proof of Theorem \ref{teolocale} read as follows:
%let $(\mathcal{O},\Omega_I)$ be the sheaf of germs of holomorphic functions  
consider the coherent analytic sheaf $(\mathcal O ^4, \Omega_I)$ of $4$-tuples of germs of holomorphic functions on the pseudoconvex domain $\Omega_I\subseteq L_I\simeq \C$. The fact that the matrix $(A,-B)$ appearing in equation \eqref{capellibis} has rank $4$ at all point $z\in \Omega_I$ means that the twelve columns $\{A^1,\ldots,A^6,B^1,\ldots, B^6\}$ generate the  stalk $\mathcal O^4_z$ of $(\mathcal O^4, \Omega_I)$ at any $z\in \Omega_I$. Theorem \ref{fasci} implies then that for any $4$-tuple $ k \in \Gamma( \Omega_I, \mathcal O^4)$ of holomorphic functions on $\Omega_I$, there exist twelve holomorphic functions $g_1,\ldots,g_{12}\in \Gamma(\Omega_I, \mathcal O)$ such that $k=g_1A^1+\cdots+g_6A^6+g_7B^1+\cdots+g_{12}B^6$. In particular, setting $k=H$ we obtain a global solution of \eqref{capellibis} and therefore a global solution  $h_1,h_2$ of equation \eqref{localeduebis} on $\Omega_I$.
%By Theorem \ref{teolocale} 
%there exists an open covering $\mathcal{U}= \{U_t \}_{t \in T}$ of $\Omega$ whose elements $U_t$ are symmetric slice domains, such that the equation
%\begin{equation}\label{globale}
%f_1*h_1+f_2*h_2=1
%\end{equation}
%has a solution on each of them.
%Consider now, for an arbitrary $J\in\s$, the slice $\Omega_J$ of $\Omega$. With the same notation used in the proof of Theorem \ref{teolocale}, equation \eqref{globale} induces by restriction a local holomorphic solution to the complex system
%\begin{equation}\label{sistemaglobale}
%\left\{\begin{array}{l}
%F_1H_1-G_1\hat K_1+F_2H_2-G_2\hat K_2=1\\
%F_1K_1+G_1\hat H_1+F_2K_2+G_2\hat H_2=0.\end{array}\right.
%\end{equation}
%on $\Omega_J$. This local solution on $\Omega_J$ is generated, via the Extension Lemma (as pointed out in the proof of Theorem \ref{teolocalebiss}), by a germ of local solution (locally defined on some slice $\Omega_I$) which belongs to $\ff^8$.
%%As pointed out in the proof of Lemma \ref{coherent}, such a solution can be represented as an element of $\ff^8$.
%Thanks to the Extension Lemma and Lemma \ref{coherent}, the sheaf $\mathcal K$ of local syzygies of $f_1, f_2$ restricted to $\Omega_J$ is coherent, and then classical arguments  (see \cite{gunningrossi}) show that system 
%\eqref{sistemaglobale} has a global holomorphic solution on $\Omega_J$.
%This is equivalent to the existence of globally defined holomorphic functions $h_1,h_2$ on $\Omega_J$ such that
%\[f_1*h_1+f_2*h_2=1\]
%on $\Omega_J$.
To conclude, applying the Extension Lemma \ref{extensionlemma}, we uniquely extend the functions $h_1,h_2$  to $\Omega$ as regular functions that satisfy
\[f_1*h_1+f_2*h_2=1\]
everywhere on $\Omega$.
\end{proof}

\section{Ideals of regular functions}

In this section we show how the proof of Theorem \ref{teoglobale} can be extended to the case of $n(\ge 2)$
regular functions with no common zeroes.

\begin{lemma}\label{ranghi}
%Let $q=x+yL\in \hh$ and 
Let $f_1, \ldots, f_n$  be $n$ regular functions in a slice symmetric domain $\Omega$ without common zeroes.
Then for any $I\in\s$ 
%and a symmetric domain $U\subseteq \Omega_I$, containing $x+yI$, such that, 
if $f_{\ell}=F_{\ell}+G_{\ell}J$ is the splitting of $f_{\ell}$ on $\Omega_I$,  for $\ell=1,\ldots,n$, then:
\begin{enumerate}
\item the rank of the $\left(2n\times {2n \choose 2}\right)$-matrix $A$ whose columns are the standard generators of the syzygies of the vector $(F_1,G_1,\ldots, F_n,G_n)$ equals $2n-1$ on $\Omega_I$;
\item the rank of  the $\left(2n\times {2n \choose 2}\right)$-matrix $B$ whose columns are the standard generators of the syzygies of the vector $(-\hat G_1,\hat F_1,\ldots, -\hat G_n,\hat F_n)$ equals $2n-1$ on $\Omega_I$;
\item the rank of the $\left(2n \times 2{2n \choose 2}\right)$-matrix $(A,-B)$ equals $2n$ on $\Omega_I$.
\end{enumerate}
\end{lemma}

\begin{proof}
%By hypothesis there exists $\ell$ such that $f_{\ell}$ does not have a spherical zero $x+y\s$ containing $q=x+yL$. We can suppose $\ell=1$. Thanks to Lemma \ref{borsuk}, as in the proof of Theorem \ref{teolocale}, we can find $I\in\s$ such that, without loss of generality, $F_1\neq 0$ and $\hat F_1\neq 0$ on a symmetric domain $U$ in $\Omega_I$ containing $x+yI$ (and $x-yI$).

Since $f_1,\ldots, f_n$ do not have common zeroes in $\Omega_I \subseteq \Omega$, the same condition is satisfied by $F_1,G_1,\ldots,F_n,G_n$. Reasoning as we did in the $n=2$ case, if $F_1(z)\neq 0$, we can reorder the columns of $A$ in such a way that all the elements in the subdiagonal are nonzero multiples of $F_1$ and all entries underneath the subdiagonal vanish.
%so that $\{A^1,\dots,A^{n-1}\}$ is a maximal subset of linearly independent columns on a neighborhood of $z$. 
If $F_1(z)=0$ and $G_1(z)\neq 0$, we can reorder (rows and columns) so that the subdiagonal is composed by nonzero multiples of $G_1$ and all the elements underneath vanish. The process can be iterated up to $G_n$.    
%we can take as a maximal subset of linearly independent columns $\{A^1,A^4,A^5\}$. If both $F_1(z)$ and $G_1(z)$ vanish, 
%we proceed analogously considering $F_2,G_2$.  
Moreover the matrix $\left(A^{2n},A^{2n+1},\ldots,A^{2n \choose 2}\right)$ has a row of zeros. This guarantees that $A$ has rank $2n-1$ on $\Omega_I$. The same argument apply to $B$ since $\hat F_1, \hat G_1,\ldots, \hat F_n, \hat G_n$ do not have common zeroes in $\Omega_I$.
%Proceeding as in proof of Theorem \ref{teolocale}, we find a symmetric domain $U$ in $L_I$ on which both $A$ and $B$ have rank $3$

To prove the third assertion, we will proceed by contradiction. Suppose that the rank of $(A,-B)$ equals $2n-1$ at $z \in \Omega_I$. Then each column of $-B$ is a linear combination of the 
%first $2n-1$ 
columns of $A$, i.e. it belongs to the syzygies of $(F_1,G_1,\ldots, F_n,G_n)$.
%and analogously each column of $A$ is a syzygy of $(-\hat G_1,\hat F_1,\ldots, -\hat G_n,\hat F_n)$.
%By mult columns of $A$ and $B$ leads to
%
%are the standard generators of the syzygies of
%$(F_1,G_1,\ldots, F_n,G_n)$ and of $(-\hat G_1,\hat F_1,\ldots, -\hat G_n,\hat F_n)$ respectively
By taking the scalar product of  each column of $B$ by $(F_1,G_1,\ldots, F_n,G_n)$, %and each column of $A$ by $(-\hat G_1,\hat F_1,\ldots, -\hat G_n,\hat F_n)$
we get ${2n \choose 2}$ equations that, as in the case $n=2$, lead to
\begin{equation}\label{sisteman}
\left\{
\begin{array}{l}
f_{ \sigma }^s = 0 \\
f^c_{\gamma}*f_{\delta}(z)=0\\
f^c_{\gamma}*f_{\delta}(\bar z)=0
\end{array}
\right.
\end{equation}
for any $\sigma,\gamma,\delta\in\{1,\ldots,n\}$, $\gamma\neq \delta$.
%Recall that $f_1^c$ is non-vanishing on $U$; hence if $f_1$ does not vanish on the sphere containing $z$ then $f_1^s(z)\neq 0$ and we find a contradiction. 
As for $n=2$, equations of the first type in system \eqref{sisteman} imply that $f_1,\ldots, f_n$ all have a (not common and not spherical) zero on the $2$-sphere generated by $z$. Following the lines of the proof of Theorem \ref{teolocale} it is possible to prove that the hypothesis that $f_1,\ldots,f_n$ do not have common zeroes leads to a contradiction.
%A computation shows that the rank of $(A,-B)$ is $2n$ on $U$.
\end{proof}

The previous lemma allows us to prove the following local result, using the same arguments of the case $n=2$.

\begin{theorem} \label{teolocalenbiss}
Let $f_1,\ldots,f_n$ be $n$ functions, regular on a symmetric slice domain $\Omega$  without common zeroes. Then for any $I\in \s$ the equation
\begin{equation}\label{localenbiss}
f_1*h_1+\cdots+f_n*h_n=1.
\end{equation}
restricted to $\Omega_I$ has local holomorphic solutions $h_1, \ldots, h_n$ in the neighborhood of any point of $\Omega_I$.
\end{theorem}

As in the proof of  Theorem \ref{teoglobale}, the consequence of Cartan Theorem B  stated in Theorem \ref{fasci} lead us to find a global solution of equation \eqref{localenbiss} on $\Omega_I$.  The Extension Lemma \ref{extensionlemma} provides a global regular solution on $\Omega$.

\begin{theorem}\label{teoglobalen}
Let $f_1, \ldots, f_n$ be regular functions on a symmetric slice domain $\Omega\subseteq\hh$, with no common zeroes in $\Omega$. Then there exist $h_1, \ldots, h_n$ regular functions  on $\Omega$ such that
\begin{equation*}
f_1*h_1+\cdots +f_n*h_n=1
\end{equation*}
on $\Omega$.
\end{theorem}

\section{Syzygies of regular functions}
We conclude the paper with a short description of the syzygies of regular functions.
Let us begin by studying the structure of the sheaf of 
%local solutions of equation \eqref{localenbiss} restricted to a suitable slice $L_I$.
local syzygies of $n$ regular functions.
\begin{theorem}\label{coherentn}
Let $f_1,\ldots, f_n$ be $n$ regular functions on a symmetric slice domain $\Omega$, with no common zeroes. 
For any $I\in \s$, and any $J\in \s, J\perp I$, let $f_\ell=F_\ell+G_\ell J$ ($\ell=1,\ldots, n$) for suitable holomorphic functions $F_\ell, G_\ell$.
%and $W\subseteq \Omega_I$ be as in Theorem \ref{teolocalenbiss}.
%let $(\ff,\Omega_I)$ be
%the sheaf of germs of holomorphic functions on $\Omega_I$, and 
If $(\mathcal{K},\Omega_I)$ is  the sheaf of germs of holomorphic solutions of the system
\begin{equation}
\left\{\begin{array}{l}
F_1H_1-G_1\hat K_1+\cdots +F_nH_n-G_n\hat K_n=0\\
F_1K_1+G_1\hat H_1+\cdots +F_nK_n+G_n\hat H_n=0.\end{array}\right.
\end{equation}
associated with 
\begin{equation}
f_1*h_1+\cdots+ f_n*h_n=0
\end{equation}
restricted to $\Omega_I$, then
%, i.e., the sheaf of germs of syzygies of $(f_1, \ldots, f_n)$ restricted to $\Omega_I$. Then
\[
(\mathcal{K},\Omega_I)\cong (\ff^{4n^2-4n}, \Omega_I)/(\ff^{4n^2-6n +2}, \Omega_I).
\]
\end{theorem}

\begin{proof}
\noindent
%As for the case $n=2$, 
Using the same notation of Lemma \ref{ranghi}, the sheaf $(\mathcal{K},\Omega_I)$ corresponds to the sheaf of germs of local solutions of the system of $2n$ equations in $2{2n \choose 2}$ unknowns
%it is direct to see that the local solutions to equation
%\begin{equation}
%f_1*h_1+ f_2*h_2=0
%\end{equation}
%restricted to a complex plane $L_I$ can be found as solutions of the system
\begin{equation}\label{syzy}
\big(\begin{array}{c  c}
 A  , & -B
 \end{array}
 \big)
 \left(\begin{array}{l}
 \alpha \\ \beta
 \end{array}
  \right)
   =0.
 \end{equation}
%With the notation used in the proof of Theorem \ref{teolocale}, it easy to see that the  matrix $(A, -B)$ has rank $4$, both matrices $A$ and $B$ have rank $3$, and $A^1,A^2,A^3$ and $B^1,B^2,B^3$ are maximal sets of linearly independent columns of $A$ and $B$ respectively.
%Therefore each element of  $\mathcal{K}$ can be written in terms of $8$ (germs of)  holomorphic functions.
%Since with our choice of $I$, the matrix $(A^1,A^2,A^3,B^{\ell})$ is invertible for some $\ell\in \{1,2,3\}$, say $\ell=1$, we can write for arbitrary germs of holomorphic functions $\alpha_4,\alpha_5,\alpha_6,\beta_2,\beta_3,\beta_4,\beta_5,\beta_6$,
%\begin{equation}
%\left\{
%\begin{array}{l}
%\alpha_{1}=\alpha_{1}(\alpha_4,\alpha_5,\alpha_6,\beta_2,\beta_3,\beta_4,\beta_5,\beta_6)\\
%\alpha_2=\alpha_2(\alpha_4,\alpha_5,\alpha_6,\beta_2,\beta_3,\beta_4,\beta_5,\beta_6)\\
%\alpha_3=\alpha_3(\alpha_4,\alpha_5,\alpha_6,\beta_2,\beta_3,\beta_4,\beta_5,\beta_6)\\
%\beta_{1}=\beta_{1}(\alpha_4,\alpha_5,\alpha_6,\beta_2,\beta_3,\beta_4,\beta_5,\beta_6).\\
%\end{array}
%\right.
%\end{equation}
Lemma \ref{ranghi} yields that we can express locally $2n$ unknowns as holomorphic functions in terms of $2{2n \choose 2}-2n=4n^2-4n$ germs of holomorphic functions.
We therefore obtain a surjective map
$$
\varphi: (\ff^{4n^2-4n}, \Omega_I) \to (\mathcal K, \Omega_I).
$$
The germ  in $(\ff^{4n^2-4n}, \Omega_I)$ associated with the vector $^t(\alpha, \beta)$, solution of \eqref{syzy},  belongs to $\ker \varphi$ if and only if
$$
A\alpha=B\beta=0,
$$
which, recalling that the rank of $A$ and $B$ equals $2n-1$, implies 
%(analogously to what happens in the case $n=2$)
%\begin{equation}
%\left\{
%\begin{array}{l}
%\alpha_1=\alpha_1(\alpha_4,\alpha_5,\alpha_6)\\
%\alpha_2=\alpha_2(\alpha_4,\alpha_5,\alpha_6)\\
%\alpha_3=\alpha_3(\alpha_4,\alpha_5,\alpha_6)\\
%\beta_1=\beta_1(\beta_4,\beta_5,\beta_6)\\
%\beta_2=\beta_2(\beta_4,\beta_5,\beta_6)\\
%\beta_3=\beta_3(\beta_4,\beta_5,\beta_6).\\
%\end{array}
%\right.
%\end{equation}
%As a consequence,
that the kernel of $\varphi$ is isomorphic to $(\ff^{4n^2-6n+2}, \Omega_I)$. Hence we conclude that $(\mathcal{K},\Omega_I)$  is isomorphic to $(\ff^{4n^2-4n}, \Omega_I)/(\ff^{4n^2-6n+2}, \Omega_I)$.
\end{proof}

In the complex case, if $f_1, \ldots , f_n$ are holomorphic functions of one complex variable with no common zeroes, then their syzygies are generated by  ${n \choose 2}$ vectors of holomorphic functions which can be constructed as follows: let $e_\ell$, $\ell = 1, \ldots, n$, be the standard basis of $\rr^n$. The generators of the syzygies are then
$$
f_re_t - f_te_r= (0,\ldots ,0, -f_t,0, \dots, 0, f_r, 0,\ldots, 0)
$$
for $1\leq r <t\leq n$, a fact which we have repeatedly used in the previous section. It is therefore natural to ask if a similar situation occurs for regular functions without common zeroes. Since the $*$-multiplication is not commutative, the immediate analogue of these syzygies does not work in this context. Natural syzygies would on the other hand be the vectors
$$
\syz(r,t):=(f_t^c*f^s_r)e_t - (f^c_r*f^s_t)e_r= (0,\ldots,0 , - f^c_r*f^s_t,0, \dots, 0, f_t^c*f^s_r, 0, \ldots, 0)
$$
for $1\leq r < t\leq n$.  In fact, Formula \eqref{uuu} implies that (see Definition \ref{R-coniugata}),
$$
f_r*(-f_r^c*f_t^s)+f_t*(f_t^c*f_r^s)=0
$$
for all $1\leq r < t\leq n$.
For $n\geq2$, as in the case of holomorphic functions, there are ${n \choose 2}$ syzygies, though Theorem
\ref{coherentn} immediately implies the following proposition.

\begin{proposition} Let $f_1, \ldots, f_n$ be regular functions on a slice symmetric domain $\Omega$ of $\hh$ with no common zeroes. Then their syzygies are locally generated by  $n-1$ vectors of regular functions.
%$$
%\syz(1,t)=(f_t^c*f^s_1)e_t- (f^c_1*f^s_t)e_1= (- f^c_1*f^s_t,\dots, f_t^c*f^s_1, \ldots, 0)
%$$
%for $1< t \leq n$. Moreover any syzygy \syz(r,t) can be obtained as
%$$
%f^s_rf^s_t\syz(r,t)= \syz(1,r
%$$
\end{proposition}
\noindent To understand this phenomenon,  we note that for any three indices $1\leq p<r<t\leq n$, we have
\begin{equation}\label{equa}
\syz(r,t)*f_p^s = \syz(p,t)*f^s_r-\syz(p,r)*f^s_t.%f_t^c*f^s_1)e_t- (f^c_1*f^s_t)e_1= (- f^c_1*f^s_t,\dots, f_t^c*f^s_1, \ldots, 0)
\end{equation}
Let us fix a sphere $S=x+y\s\subseteq \Omega$. If one of the functions $f_p, f_r, f_t$ never vanishes on $S$, assume $f_p$, then \eqref{equa} immediately shows that $\syz(r,t)$ is a combination with regular coefficients of   $\syz(p,t)$ and $\syz(p,r)$
\begin{equation}\label{equa1}
\syz(r,t)= \syz(p,t)*f^s_r*(f^s_p)^{-1}-\syz(p,r)*f^s_t *(f^s_p)^{-1}.
\end{equation}
If all $f_p, f_r, f_t$ have a zero on $S$, without loss of generality, we can assume that $f_p$ has the lesser order (for the notion of order of a zero see, e.g., \cite{libroGSS}). Then again \eqref{equa1} can be used to represent $\syz(r,t)$ locally.

\begin{oss} {\rm It therefore appears that the reason why we can reduce to $n-1$ the number of syzygies is a consequence of Remark \ref{oss}, namely the fact that a (isolated, non real) zero of a regular function $f$ generates a sphere of zeroes for $f^s$ and a sphere of poles for its reciprocal $f^{-*}$.
}\end{oss}

 \section*{Acknowledgements}
The first two authors acknowledge the support of G.N.S.A.G.A. of INdAM and of Italian MIUR (Research Projects PRIN ``Real and complex manifolds: geometry, topology and harmonic analysis'', FIRB ``Differential geometry and geometric function theory''). They express their gratitude to Chapman University, where a portion of this work was carried out. The second author is partly supported by the SIR project ``Analytic aspects of complex and hypercomplex geometry''
of the Italian MIUR.

\end{document}